\documentclass[twoside,reqno]{amsart}

\usepackage{amsmath,amssymb,amsbsy,amstext,amsxtra,latexsym}
\usepackage{graphicx}
\usepackage{subfig}
\usepackage{enumerate}
\usepackage[all]{xy}
\usepackage{cancel}
\usepackage{chngpage}
\usepackage{hyperref}

\newtheorem{theorem}{Theorem}[section]
\newtheorem*{maintheorem}{Theorem}
\newtheorem{lemma}[theorem]{Lemma}
\newtheorem{proposition}[theorem]{Proposition}

\theoremstyle{definition}
\newtheorem*{definition}{Definition}

\theoremstyle{remark}

\newtheorem*{remark}{Remark}

\numberwithin{equation}{section}

\newcommand{\uc}[1]{\ensuremath \overset{#1}{\circ}}

\newcommand{\blup}[2]{\ensuremath #1 \sharp #2 \overline{\mathbb{CP}}^2}

\DeclareMathOperator{\image}{im}
\DeclareMathOperator{\id}{id}
\DeclareMathOperator{\Ext}{Ext}

\def\sheaf#1{\ensuremath \mathcal#1}

\begin{document}

\title[Surfaces of general type with $p_g=1$ and $q=0$]{Surfaces
of general type with $p_g=1$ and $q=0$}

\author{Heesang Park}

\address{School of Mathematics, Korea Institute for Advanced Study, Seoul 130-722, Korea}

\email{hspark@kias.re.kr}

\author{Jongil Park}

\address{Department of Mathematical Sciences, Seoul National University, Seoul 151-747, Korea \&
         Korea Institute for Advanced Study, Seoul 130-722, Korea}

\email{jipark@snu.ac.kr}

\author{Dongsoo Shin}

\address{Department of Mathematics, Chungnam National University, Daejeon 305-764, Korea}

\email{dsshin@cnu.ac.kr}




\begin{abstract}
In this paper we construct a new family of simply connected minimal complex surfaces of general type with $p_g=1$, $q=0$, and $K^2=3, 4, 5, 6, 8$ using a $\mathbb{Q}$-Gorenstein smoothing theory. We also reconstruct minimal complex surfaces of general type with $p_g=1$, $q=0$, and $K^2=1, 2$ using the same method.
\end{abstract}

\maketitle

\section{Introduction}

 In the geography of minimal complex surfaces of general type, one of the fundamental problems
 is to find a new family of simply connected minimal surfaces with given topological invariants
 such as $p_g$, $q$, and $K^2$.
 This geography problem has been studied extensively by algebraic geometers and topologists
 for a long time, so that various families of surfaces of general type have been constructed
 (\cite{BHPV}, Chapter VII).
 Nevertheless, in the case of $p_g=1$ and $q=0$, only a few examples are known (see below)
 and there were no such simply connected examples with $K^2 \geq 3$ previously known.
 Note that this class of surfaces has been drawn attention because they provide counterexamples
 to the Torelli problems.

 One the other hand, complex surfaces of general type with $p_g=1$ and $q=0$ also provide exotic smooth structures on the topological $4$-manifolds $3\mathbb{CP}^2 \sharp n \overline{\mathbb{CP}}^2$.
In fact, many examples of exotic $3\mathbb{CP}^2 \sharp n \overline{\mathbb{CP}}^2$ for $n=5$ or $7 \le n \le 19$ were constructed via various surgery techniques. For example, Gompf~\cite{Gompf} constructed exotic $3\mathbb{CP}^2 \sharp n \overline{\mathbb{CP}}^2$ for $14 \le n \le 18$ by a symplectic sum, B.D. Park~\cite{BDPark} constructed exotic $3\mathbb{CP}^2 \sharp n \overline{\mathbb{CP}}^2$ for $10 \le n \le 13$ via knot surgery and symplectic sum, and
 Stipsicz and Szab\'{o}~\cite{Stipsicz-Szabo}, and the second author~\cite{Jpark} constructed exotic $3\mathbb{CP}^2 \sharp 9 \overline{\mathbb{CP}}^2$ and $3\mathbb{CP}^2 \sharp 8\overline{\mathbb{CP}}^2$ using rational blowdowns, respectively.
 However it was not known whether those exotic $3\mathbb{CP}^2 \sharp n \overline{\mathbb{CP}}^2$'s constructed so far admit complex structures.

 All known surfaces of general type with $p_g=1$ and $q=0$ are constructed by classical methods: Quotient and covering.  Kynev~\cite{Kynev} constructed a surface with $K^2=1$ as a quotient of the Fermat sextic in $\mathbb{CP}^3$ by a suitable action of a group of order $6$. According to Catanese~\cite{Catanese}, all minimal surfaces of general type with $p_g=1$ and $K^2=1$ are diffeomorphic and simply connected. Catanese and Debarre~\cite{CD} constructed surfaces with $K^2=2$ by double coverings of the projective plane $\mathbb{CP}^2$ or smooth minimal K3 surfaces. In fact, they classified such surfaces into five classes according to the degree and the image of the bicanonical map. Four of them are simply connected and the other one has a torsion $\mathbb{Z}/2\mathbb{Z}$. Todorov~\cite{Todorov} also constructed non-simply connected surfaces with $2 \le K^2 \le 8$ by considering double covers of K3 surfaces. Note that all his examples with $3 \le K^2 \le 8$ have big fundamental groups.

The main result of this paper is the following.

\begin{maintheorem}
 There are simply connected minimal complex surfaces of general type with $p_g=1$, $q=0$,
 and $K^2=1, 2,\dotsc,6, 8$ and a minimal complex surface of general type with
 $p_g=1$, $q=0$, $K^2=2$, and $H_1 = \mathbb{Z}/2\mathbb{Z}$.
\end{maintheorem}

In order to construct such surfaces with $1 \le K^2 \le 6$, we take a similar strategy in Y. Lee and J. Park~\cite{Lee-Park-K^2=2}. We blow up an elliptic K3 surface in a suitable set of points so that we obtain a surface with a very special configuration of rational curves. Inside this configuration we find some disjoint chains which can be contracted to special quotient singularities. These singularities admit a local $\mathbb{Q}$-Gorenstein smoothing, which is a smoothing whose relative canonical class is $\mathbb{Q}$-Cartier. And then, we prove that these local smoothings can be glued to a global $\mathbb{Q}$-Gorenstein smoothing of the whole singular surface by showing that the obstruction space of a global smoothing is zero. Finally it is not difficult to show that a general fiber of the smoothing is the desired surface. The key ingredient of the construction is to develop a new method for proving that the obstruction space is zero because the method in Y. Lee and J. Park~\cite{Lee-Park-K^2=2} for a computation of the obstruction space cannot be applied to our cases.

For constructing a simply connected surface with $K^2=8$, we blow up a K3 surface $\overline{Y}$ in a suitable set of points so that we obtain a surface with some special disjoint linear chains of rational curves which can be contracted to singularities class $T$ on a singular surface $\overline{X}$ with  $H^2(\sheaf{T_{\overline{X}}}) \neq 0$. In order to prove the existence of a global $\mathbb{Q}$-Gorenstein smoothing of $\overline{X}$, we apply the cyclic covering trick developed in Y. Lee and J. Park~\cite{Lee-Park-Horikawa}. The cyclic covering trick says that, if a cyclic covering $\pi : V \to W$ of singular surfaces satisfies certain conditions and the base $W$ has a $\mathbb{Q}$-Gorenstein smoothing, then the cover $V$ has also a $\mathbb{Q}$-Gorenstein smoothing. The main ingredient is that we construct an unramified double covering $\overline{\pi} : \overline{X} \to X$ to a singular surface $X$ constructed in a recent paper \cite{Park} of the first author. It is a main result of H. Park~\cite{Park} that the singular surface $X$ has a global $\mathbb{Q}$-Gorenstein smoothing and a general fiber $X_t$ of the smoothing of $X$ is a surface of general type with $p_g=0$, $K^2=4$, and $\pi_1 = \mathbb{Z}/2\mathbb{Z}$. We show that the double covering $\overline{\pi} : \overline{X} \to X$ satisfies all the conditions of the cyclic covering trick; hence there is a global $\mathbb{Q}$-Gorenstein smoothing of $\overline{X}$. Then it is not difficult to show that a general fiber $\overline{X}_t$ of the smoothing of $\overline{X}$ is the desired surface.

This paper is organized as follows. In Section~\ref{section:construction} we reconstruct complex surfaces with $p_g=1$, $q=0$ and $K^2=2$: A simply connected surface and a surface with $H_1 = \mathbb{Z}/2\mathbb{Z}$. We prove in Section~\ref{section:obstruction=0} that the obstruction spaces of global smoothings of the singular surfaces constructed in Section~\ref{section:construction} are zero. In Section~\ref{section:Examples} and Section~\ref{section:K^2=8}, we construct various examples of simply connected surfaces with $p_g=1$, $q=0$ and $K^2=1, 2,\dotsc, 6, 8$.

\subsubsection*{Acknowledgement}
The authors would like to thank Professor Yongnam Lee for some valuable comments on the first draft of this article. Jongil Park was supported by Basic Science Research Program through the National Research Foundation of Korea (NRF) Grant funded by the Korean Government (2009-0093866). He also holds a joint appointment at Korea Institute for Advanced Study and at Research Institute of Mathematics in Seoul National University. Dongsoo Shin was supported by Basic Science Research Program through the National Research Foundation of Korea (NRF) grant funded by the Korean Government (2010-0002678).

\section{Surfaces with $p_g=1$, $q=0$, and $K^2=2$}
\label{section:construction}

In this section we construct a simply connected minimal surface of general type with $p_g=1$, $q=0$, and $K^2=2$. We start with a special rational elliptic surface $E(1)$. Let $L_1$, $L_2$, $L_3$, and $A$ be lines in $\mathbb{CP}^2$ and let $B$ be a smooth conic in $\mathbb{CP}^2$ intersecting as in Figure~\ref{figure:K2-E}(A). We consider a pencil of cubics $\{\lambda(L_1+L_2+L_3) + \mu(A+B) \mid [\lambda:\mu] \in \mathbb{CP}^1 \}$ in $\mathbb{CP}^2$ generated by two cubic curves $L_1+L_2+L_3$ and $A+B$, which has $4$ base points, say, $p$, $q$, $r$ and $s$. In order to obtain an elliptic fibration over $\mathbb{CP}^1$ from the pencil, we blow up three times at $p$ and $r$, respectively, and twice at $s$, including infinitely near base-points at each point. We perform one further blowing-up at the base point $q$. By blowing-up totally nine times, we resolve all base points (including infinitely near base-points) of the pencil and we then get a rational elliptic surface $E(1)$ with an $I_8$-singular fiber, an $I_2$-singular fiber, two nodal singular fibers, and four sections; Figure~\ref{figure:K2-E}(B).

\begin{figure}[hbtb]
\centering
\subfloat[A pencil]{\includegraphics[scale=1]{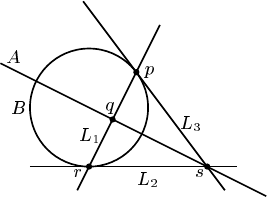}} \qquad \subfloat[$E(1)$]{\includegraphics[scale=0.7]{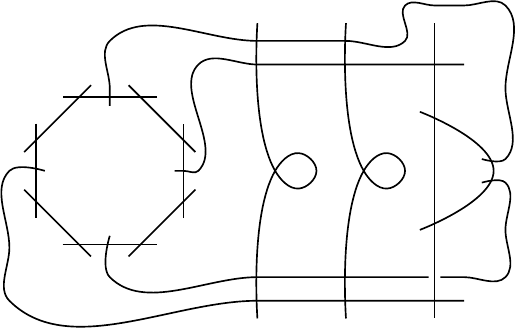}}
\caption{A rational elliptic surface $E(1)$}
\label{figure:K2-E}
\end{figure}

Let $Y$ be a double cover of the rational elliptic surface $E(1)$ branched along two general fibers. Then the surface $Y$ is an elliptic K3 surface with two $I_8$-singular fibers, two $I_2$-singular fibers, four nodal singular fibers, and four sections; Figure~\ref{figure:K3-surface-Y}(A). In the following construction we use only one $I_8$-singular fiber, one nodal singular fiber, and three sections; Figure~\ref{figure:K3-surface-Y}(B).

\begin{figure}[hbtb]
\centering
\subfloat[The fibration structure of $Y$]{\includegraphics[scale=0.6]{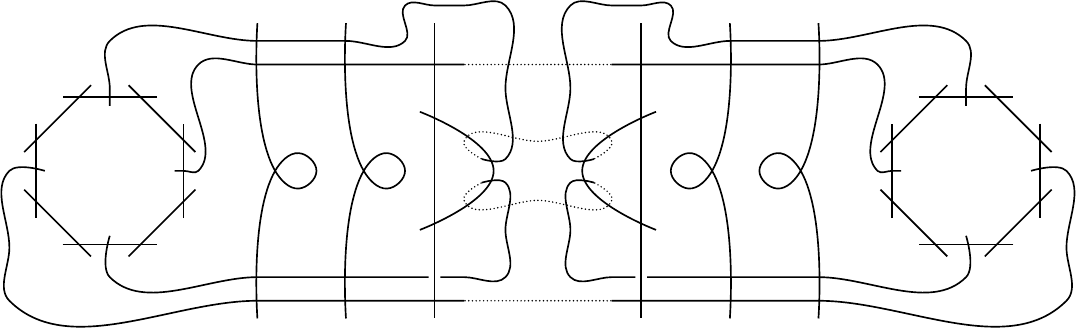}} \\ \subfloat[A part of $Y$]{\includegraphics[scale=0.7]{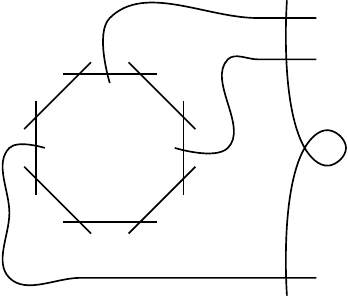}}
\caption{An elliptic K3 surface $Y$}
\label{figure:K3-surface-Y}
\end{figure}

Let $\tau : V \to Y$ be the blowing-up at the node of the nodal singular fiber and let $E$ be the exceptional divisor of $\tau$. Since $K_Y = 0$, we have $K_V = E$. Let $D = D_1 + \dotsb + D_6$ be the part of the $I_8$-singular fiber. Let $S_i$ ($i=1,2,3$) be the sections of the fibration $V \to \mathbb{CP}^1$ and set $S = S_1 + S_2 + S_3$. Let $F$ be the proper transform of the nodal fiber of the K3 surface $Y$; Figure~\ref{figure:K2-sc-V}.

\begin{figure}[hbtb]
\centering
\includegraphics[scale=0.7]{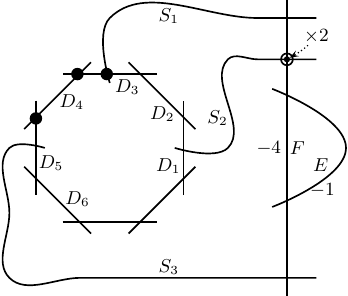}
\caption{A surface $V = \blup{Y}{}$}
\label{figure:K2-sc-V}
\end{figure}

We blow up the surface $V$ three times totally at the three marked points $\bullet$ and blow up twice at the marked point $\bigodot$. We then get a surface $Z$; Figure~\ref{figure:K2-sc-Z}. There exist three disjoint linear chains of ${\mathbb{CP}}^1$'s in $Z$:
\begin{equation*}
\uc{-3}-\uc{-6}-\uc{-2}-\uc{-3}-\uc{-2}, \quad \uc{-4}-\uc{-2}-\uc{-2}-\uc{-3}-\uc{-2}, \quad \uc{-4}.
\end{equation*}

\begin{figure}[hbtb]
\centering
\includegraphics[scale=0.7]{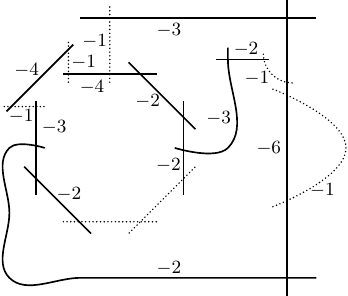}
\caption{A surface $Z = \blup{Y}{6}$}
\label{figure:K2-sc-Z}
\end{figure}

\subsection*{Main construction}

By applying $\mathbb{Q}$-Gorenstein smoothing theory to the surface $Z$ as in Y. Lee and J. Park~\cite{Lee-Park-K^2=2} and the authors~\cite{PPS-K3, PPS-K4}, we construct a complex surface of general type with $p_g=1$, $q=0$, and $K^2=2$. That is, we first contract the three chains of $\mathbb{CP}^1$'s from the surface $Z$ so that it produces a normal projective surface $X$ with three permissible singular points. In Section~\ref{section:obstruction=0} we will show that the singular surface $X$ has a global $\mathbb{Q}$-Gorenstein smoothing. Let $X_t$ be a general fiber of the $\mathbb{Q}$-Gorenstein smoothing of $X$. Since $X$ is a (singular) surface with $p_g=1$, $q=0$, and $K^2=2$, by applying general results of complex surface theory and $\mathbb{Q}$-Gorenstein smoothing theory, one may conclude that a general fiber $X_t$ is a complex surface of general type with $p_g=1$, $q=0$, and $K^2=2$. Furthermore, it is not difficult to show that a general fiber $X_t$ is minimal and simply connected by using a similar technique in Y. Lee and J. Park~\cite{Lee-Park-K^2=2} and the authors~\cite{PPS-K3, PPS-K4}.

\begin{remark}
Catanese and Debarre~\cite{CD} proved that surfaces of general type with $p_g=1$, $q=0$, and $K^2=2$ are divided into five types according to the degree and the image of the bicanonical map. Four of them are simply connected and the other one has a torsion $\mathbb{Z}/2\mathbb{Z}$. It is an interesting problem to determine in which class the simply connected example constructed in this section is contained. We leave this question for the future research.
\end{remark}

\subsection{An example with $K^2=2$ and $H_1 = \mathbb{Z}/2\mathbb{Z}$}
\label{sec:h_1=z_2}

We construct a minimal complex surface of general type with $p_g=1$, $q=0$, $K^2=2$, and $H_1 = \mathbb{Z}/2\mathbb{Z}$. Let $A$, $L_i$ ($i=1,2,3$) be lines on the projective plane $\mathbb{CP}^2$ and $B$ a nonsingular conic on $\mathbb{CP}^2$ which intersect as in Figure~\ref{figure:K2-z2-E}(A). Consider a pencil of cubics generated by the two cubics $A+B$ and $L_1+L_2+L_3$. By resolving the base points of the pencil of cubics including infinitely near base-points, we obtain a rational elliptic surface $E(1)$ with an $I_7$-singular fiber, an $I_2$-singular fiber, an cusp singular fiber, a nodal singular fiber, and five sections as in Figure~\ref{figure:K2-z2-E}(B).

\begin{figure}[hbtb]
\centering
\subfloat[A pencil]{\includegraphics[scale=1]{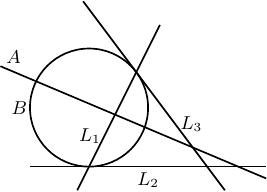}} \qquad
\subfloat[$E(1)$]{\includegraphics[scale=0.7]{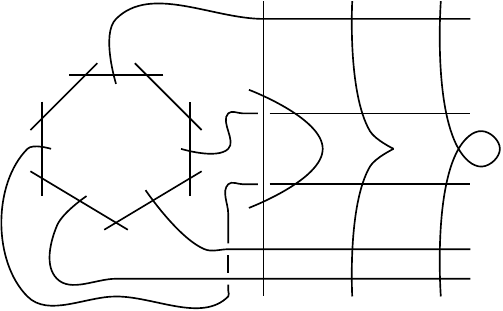}}
\caption{A rational elliptic surface $E(1)$ for $K^2=2$ and $H_1=\mathbb{Z}/2\mathbb{Z}$}
\label{figure:K2-z2-E}
\end{figure}

Let $Y$ be a double cover of the rational elliptic surface $E(1)$ branched along two general fibers $F_1$ and $F_2$ near the cusp singular fiber. Then $Y$ is an elliptic K3 surface with two $I_7$-singular fibers, two $I_2$-singular fibers, two cusp singular fibers, two nodal singular fibers, and five sections. We use only two $I_7$-singular fibers, two $I_2$-singular fibers, and two sections; Figure~\ref{figure:K2-z2}(A). We blow up eight times totally at the eight marked points $\bullet$. We then get a surface $Z$; Figure~\ref{figure:K2-z2}(B). There exist six disjoint linear chains of ${\mathbb{CP}}^1$'s in $Z$:
\begin{gather*}
\uc{-5}-\uc{-3}-\uc{-2}-\uc{-2}, \quad \uc{-5}-\uc{-3}-\uc{-2}-\uc{-2}, \quad \uc{-3}-\uc{-2}-\uc{-3}, \quad \uc{-3}-\uc{-2}-\uc{-3}, \quad \uc{-4}, \quad \uc{-4}
\end{gather*}

\begin{figure}[hbtb]
\centering \subfloat[$Y$]{\includegraphics[scale=0.7]{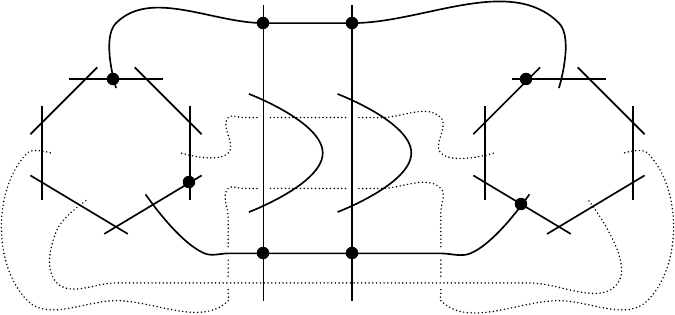}} \\
\centering \subfloat[$Z = \blup{Y}{8}$]{\includegraphics[scale=0.7]{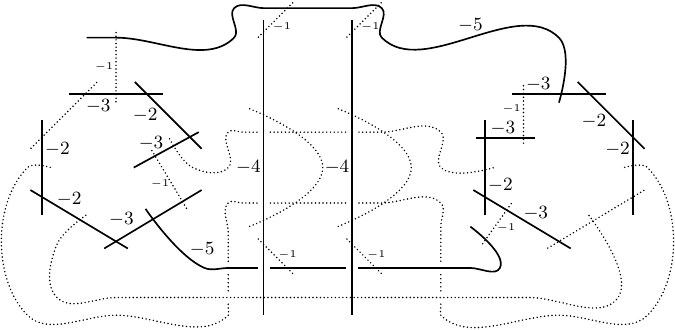}}
\caption{An example with $H_1 = \mathbb{Z}/2\mathbb{Z}$}
\label{figure:K2-z2}
\end{figure}

We now briefly explain how to prove that the complex surface obtained by the above configuration has $H_1 = \mathbb{Z}/2\mathbb{Z}$. For proving this one may apply a similar method in Y. Lee and J. park~\cite{Lee-Park-H1=Z2} and the authors~\cite{PPS-Z2}. The key part of the method is finding a basis for the free abelian group $H_2(Z, \mathbb{Z})$ and calculating how the elements of the basis intersect with the contracted rational curves. Then the proof of being $H_1 = \mathbb{Z}/2\mathbb{Z}$ follows because the configuration of the contracted rational curves are very special. A basis for $H_2(Z, \mathbb{Z})$ in our case can be obtained as follows: Since the K3 surface $Y$ may be regarded as a $4$-manifold obtained by applying a fiber sum surgery from the rational elliptic surface $E(1)$ along two general fibers near the cusp singular fiber, the homology $H_2(Z, \mathbb{Z})$ are generated by 18 spheres induced from a basis for $H_1(E(1),\mathbb{Z})$, four new homology elements emerging during the fiber sum surgery, and the eight $-1$-curves induced from eight blowing-ups on $Y$; cf.~\cite[p.\thinspace70]{Gompf-Stipsicz}. It is not difficult to calculate how the 18 spheres induced from a basis for $H_1(E(1),\mathbb{Z})$ intersect with the contracted rational curves because a basis for $H_1(E(1),\mathbb{Z})$ can be represented by linear combinations of $-2$-curves consisting of the $I_7$-singular fiber and the five sections. Furthermore the four new homology elements emerging during the fiber sum surgery do not intersect the contracted rational curves. Hence one can apply a similar method in Y. Lee and J. park~\cite{Lee-Park-H1=Z2} and the authors~\cite{PPS-Z2} to show that the complex surface obtained by the above configuration has $H_1 = \mathbb{Z}/2\mathbb{Z}$.

\section{The Existence of $\mathbb{Q}$-Gorenstein smoothings}
\label{section:obstruction=0}

This section is devoted to the proof of the following theorem.

\begin{theorem}
\label{theorem:Main-Q-Gorenstein}
The singular surface $X$ constructed in the main construction in Section~\ref{section:construction} has a global $\mathbb{Q}$-Gorenstein smoothing.
\end{theorem}

For this, we first briefly review some basic facts on $\mathbb{Q}$-Gorenstein smoothing theory for normal projective surfaces with special quotient singularities (refer to Y. Lee and J. Park~\cite{Lee-Park-K^2=2} for details).

\begin{definition}
Let $X$ be a normal projective surface with quotient singularities. Let $\mathcal{X}\to\Delta$ (or $\mathcal{X}/\Delta$) be a flat family of projective surfaces over a small disk $\Delta$. The one-parameter family of surfaces $\mathcal{X}\to\Delta$ is called a \emph{$\mathbb{Q}$-Gorenstein smoothing} of $X$ if it satisfies the following three conditions;

\begin{enumerate}[(i)]
\item the general fiber $X_t$ is a smooth projective surface,

\item the central fiber $X_0$ is $X$,

\item the canonical divisor $K_{\mathcal{X}/\Delta}$ is $\mathbb{Q}$-Cartier.
\end{enumerate}
\end{definition}

A $\mathbb{Q}$-Gorenstein smoothing for a germ of a quotient singularity $(X_0, 0)$ is defined similarly. A quotient singularity which admits a $\mathbb{Q}$-Gorenstein smoothing is called a \emph{singularity of class T}.

\begin{proposition}[\cite{KSB, Manetti}]
Let $(X_0, 0)$ be a germ of two dimensional quotient singularity. If $(X_0, 0)$ admits a $\mathbb{Q}$-Gorenstein smoothing over the disk, then $(X_0, 0)$ is either a rational double point or a cyclic quotient singularity of type \mbox{$\frac{1}{dn^2}(1, dna-1)$} for some integers $a, n, d$ with $a$ and $n$ relatively prime.
\end{proposition}

\begin{proposition}[\cite{KSB, Manetti}]

\begin{enumerate}[(1)]
\item The singularities ${\overset{-4}{\circ}}$ and ${\overset{-3}{\circ}}-{\overset{-2}{\circ}}-{\overset{-2}{\circ}}-\cdots- {\overset{-2}{\circ}}-{\overset{-3}{\circ}}$ are of class $T$.

\item If the singularity ${\overset{-b_1}{\circ}}-\cdots-{\overset{-b_r}{\circ}}$ is of class $T$, then so are
\[{\overset{-2}{\circ}}-{\overset{-b_1}{\circ}}-\cdots-{\overset{-b_{r-1}} {\circ}}- {\overset{-b_r-1}{\circ}} \quad\text{and}\quad  {\overset{-b_1-1}{\circ}}-{\overset{-b_2}{\circ}}-\cdots- {\overset{-b_r}{\circ}}-{\overset{-2}{\circ}}.\]

\item Every singularity of class $T$ that is not a rational double point can be obtained by starting with one of the singularities described in $(1)$ and iterating the steps described in $(2)$.
\end{enumerate}
\end{proposition}

Let $X$ be a normal projective surface with singularities of class $T$. Due to the result of Koll\'ar and Shepherd-Barron~\cite{KSB}, there is a $\mathbb{Q}$-Gorenstein smoothing locally for each singularity of class $T$ on $X$. The natural question arises whether this local $\mathbb{Q}$-Gorenstein smoothing can be extended over the global surface $X$ or not. Roughly geometric interpretation is the following: Let $\cup_{\alpha} V_{\alpha}$ be an open covering of $X$ such that each $V_{\alpha}$ has at most one singularity of class $T$. By the existence of a local $\mathbb{Q}$-Gorenstein smoothing, there is a $\mathbb{Q}$-Gorenstein smoothing $\mathcal{V}_{\alpha}/\Delta$. The question is if these families glue to a global one. The answer can be obtained by figuring out the obstruction map of the sheaves of deformation $T^i_X=Ext^i_X(\Omega_X,\sheaf{O_X})$ for $i=0,1,2$. For example, if $X$ is a smooth surface, then $T^0_X$ is the usual holomorphic tangent sheaf $T_X$ and $T^1_X=T^2_X=0$. By applying the standard result of deformations to a normal projective surface with quotient singularities, we get the following

\begin{proposition}[{Wahl~\cite[\S4]{Wahl}}]

Let $X$ be a normal projective surface with quotient singularities. Then
\begin{enumerate}[(1)]
\item The first order deformation space of $X$ is represented by the global Ext 1-group $\mathbb{T}^1_X=\Ext^1_X(\Omega_X, \sheaf{O_X})$.

\item The obstruction lies in the global Ext 2-group $\mathbb{T}^2_X=\Ext^2_X(\Omega_X, \sheaf{O_X})$.
\end{enumerate}
\end{proposition}

Furthermore, by applying the general result of local-global spectral sequence of ext sheaves  to deformation theory of surfaces with quotient singularities so that $E_2^{p, q}=H^p(T^q_X) \Rightarrow \mathbb{T}^{p+q}_X$, and by $H^j(T^i_X)=0$ for $i, j\ge 1$, we also get

\begin{proposition}[Manetti~\cite{Manetti}, Wahl~\cite{Wahl}]
Let $X$ be a normal projective surface with quotient singularities. Then

\begin{enumerate}[(1)]
\item We have the exact sequence
\[0\to H^1(T^0_X)\to \mathbb{T}^1_X\to \ker [H^0(T^1_X)\to H^2(T^0_X)]\to 0\]
where $H^1(T^0_X)$ represents the first order deformations of $X$ for which the singularities remain locally a product.

\item If $H^2(T^0_X)=0$, every local deformation of the singularities may be globalized.
\end{enumerate}
\end{proposition}

As mentioned above, there is a local $\mathbb{Q}$-Gorenstein smoothing for each singularity of $X$ due to the result of Koll\'ar and Shepherd-Barron~\cite{KSB}. Hence it remains to show that every local deformation of the singularities can be globalized. Note that the following proposition tells us a sufficient condition for the existence of a global $\mathbb{Q}$-Gorenstein smoothing of $X$.

\begin{proposition}[Y. Lee and J. Park~\cite{Lee-Park-K^2=2}]
\label{proposition:sufficient_H^2=0}
Let $X$ be a normal projective surface with singularities of class $T$. Let $\pi : V \to X$ be the minimal resolution and let $A$ be the reduced exceptional divisor. Suppose that $H^2(T_V(-\log{A}))=0$. Then $H^2(T^0_X)=0$ and there is a $\mathbb{Q}$-Gorenstein smoothing of $X$.
\end{proposition}

Furthermore, the proposition above can be easily generalized to any log resolution of $X$ by keeping the vanishing of cohomologies under blowing up at the points. It is obtained by the following well-known result.

\begin{proposition} [{Flenner and Zaidenberg~\cite[\S1]{FZ}}]
Let $V$ be a nonsingular surface and let $A$ be a simple normal crossing divisor in $V$. Let $f : V' \to V$ be a blowing up of $V$ at a point p of $A$. Set $A'=f^{-1}(A)_{red}$. Then $h^2(T_{V'}(-\log{A'}))=h^2(T_V(-\log{A}))$.
\end{proposition}

Therefore Theorem~\ref{theorem:Main-Q-Gorenstein} follows from Proposition~\ref{proposition:sufficient_H^2=0} and the following proposition:

\begin{proposition}
\label{proposition:H^2=0}
$H^2(T_V(-\log(D + S + F))) = H^0(\Omega_V(\log(D + S + F))(E)) = 0$.
\end{proposition}

The idea of the proof is as follows: There is an exact sequence of sheaves
\begin{equation*}
0 \to \Omega_V \to \Omega_V(\log(D+S+F)) \to \bigoplus_{i=1}^{6} \sheaf{O_{D_i}} \oplus \bigoplus_{i=1}^{3} \sheaf{O_{S_i}} \oplus \sheaf{O_F} \to 0.
\end{equation*}
By tensoring $\sheaf{O_V}(E)$, we have an exact sequence
\begin{equation*}
0 \to \Omega_V(E) \to \Omega_V(\log(D+S+F))(E) \to \bigoplus_{i=1}^{6} \sheaf{O_{D_i}} \oplus \bigoplus_{i=1}^{3} \sheaf{O_{S_i}} \oplus \sheaf{O_F}(E) \to 0.
\end{equation*}
Since $H^0(\Omega_V(E))=0$, the proof of Proposition~\ref{proposition:H^2=0} is done if we show that the connecting homomorphism
\begin{equation*}
\bigoplus_{i=1}^{6} H^0(\sheaf{O_{D_i}}) \oplus \bigoplus_{i=1}^{3} H^0(\sheaf{O_{S_i}}) \oplus H^0(\sheaf{O_F}(E)) \to H^1(\Omega_{V}(E))
\end{equation*}
is injective.

The proof of Proposition~\ref{proposition:H^2=0} begins with the following Lemma. We have a commutative diagram
\begin{equation}\label{equation:c_1-d_1-diagram}
\xymatrix{
\displaystyle  \bigoplus_{i=1}^{6} H^0(\sheaf{O_{D_i}}) \oplus \bigoplus_{i=1}^{3} H^0(\sheaf{O_{S_i}}) \oplus H^0(\sheaf{O_F}) \ar[r]^(.7){c_1} \ar[d]& H^1(\Omega_V) \ar[d]^{\beta_1} \\ %
\displaystyle  \bigoplus_{i=1}^{6} H^0(\sheaf{O_{D_i}}) \oplus \bigoplus_{i=1}^{3} H^0(\sheaf{O_{S_i}}) \oplus H^0(\sheaf{O_F}(E)) \ar[r]^(.7){d_1} & H^1(\Omega_V(E))
}
\end{equation}

\begin{lemma}
\label{lemma:composition}
The composition map
\begin{equation*}
\beta_1 \circ c_1 : \bigoplus_{i=1}^{6} H^0(\sheaf{O_{D_i}}) \oplus \bigoplus_{i=1}^{3} H^0(\sheaf{O_{S_i}}) \oplus H^0(\sheaf{O_F}) \to H^1(\Omega_{V}(E))
\end{equation*}
is injective.
\end{lemma}

\begin{proof}
Note that the map $c_1$ is the first Chern class map. Since the intersection matrix, whose entries are the intersection numbers of $D_i$ ($i=1,\dotsc,6$), $S_j$ ($j=1,2,3$), and $F$, is invertible, their images by the map $c_1$ are linearly independent. Therefore the map $c_1$ is injective.

Consider the commutative diagram
\begin{equation*}
\xymatrix{%
& 0 \ar[d] \\ %
& H^0(\Omega_V \otimes \sheaf{O_E}(E)) \ar[d]^{\delta} \\ %
\displaystyle  \bigoplus_{i=1}^{6} H^0(\sheaf{O_{D_i}}) \oplus \bigoplus_{i=1}^{3} H^0(\sheaf{O_{S_i}}) \oplus H^0(\sheaf{O_F}) \ar[r]^(.7){c_1} \ar[d]& H^1(\Omega_V) \ar[d]^{\beta_1} \\ %
\displaystyle  \bigoplus_{i=1}^{6} H^0(\sheaf{O_{D_i}}) \oplus \bigoplus_{i=1}^{3} H^0(\sheaf{O_{S_i}}) \oplus H^0(\sheaf{O_F}(E)) \ar[r]^(.7){d_1} & H^1(\Omega_V(E))
}
\end{equation*}
where the vertical sequence is induced from the exact sequence
\begin{equation*}
0 \to H^0(\Omega_V \otimes \sheaf{O_E}(E)) \xrightarrow{\delta} H^1(\Omega_V) \xrightarrow{\beta_1} H^1(\Omega_V(E)) \xrightarrow{\gamma} H^1(\Omega_V \otimes \sheaf{O_E}(E)) \to 0.
\end{equation*}

\noindent \textit{Claim}: The connecting homomorphism $\delta : H^0(\Omega_V \otimes \sheaf{O_E}(E)) \to H^1(\Omega_V)$ is the first Chern class map of $\sheaf{O_V}(E)$: Since $H^0(\Omega_E \otimes \sheaf{O_E}(E)) = 0$, we have an isomorphism
\begin{equation*}
H^0(\sheaf{O_E}(-E) \otimes \sheaf{O_E}(E)) \xrightarrow{\cong} H^0(\Omega_V \otimes \sheaf{O_E}(E)).
\end{equation*}
Here the above map is given by $z_{\alpha} \otimes \frac{1}{z_{\alpha}} \mapsto \left\{ \frac{d z_{\alpha}}{z_{\alpha}} \right\}$, where $z_{\alpha}$ is a local equation of $E$. Therefore the connecting homomorphism $\delta: H^0(\Omega_V \otimes \sheaf{O_E}(E)) \to H^1(\Omega_V)$ is given by
\begin{equation*}
\left\{ \frac{d z_{\alpha}}{z_{\alpha}} \right\} \mapsto \left\{ \frac{d z_{\alpha}}{z_{\alpha}} - \frac{d z_{\beta}}{z_{\beta}} \right\} = \left\{ d \left( \dfrac{z_{\alpha}}{z_{\beta}} \right)/\left( \dfrac{z_{\alpha}}{z_{\beta}} \right) \right\},
\end{equation*}
which is the first Chern class map of $\sheaf{O_V}(E)$. This proves the claim.

Since $E$ is the exceptional divisor, it is independent of the other divisors in $H^1(\Omega_V)$. Therefore $\image{c_1} \cap \image{\delta} = 0$; hence the composition map $\beta_1 \circ c_1$ is injective.
\end{proof}

We now concentrate on the following restriction map $d_1'$ of the map $d_1$ in \eqref{equation:c_1-d_1-diagram}:
\begin{equation*}
d_1' : H^0(\sheaf{O_F}(E)) \to H^1(\Omega_V(E)),
\end{equation*}
which is also regarded as a connecting homomorphism induced from the exact sequence
\begin{equation*}
0 \to \Omega_V \to \Omega_V(\log{F}) \to \sheaf{O_F} \to 0
\end{equation*}
tensored by $\sheaf{O_V}(E)$.

\begin{lemma}
We have $H^0(\Omega_V(\log{F})(E)) = 0$.
\end{lemma}

\begin{proof}
Let $C$ be a general fiber of the elliptic fibration $\phi : Y \to \mathbb{CP}^1$. By the projection formula we have
\begin{equation*}
\begin{split}
H^0(\Omega_V(\log{F})(E)) &\subseteq H^0(\Omega_V(F + E)) = H^0(\Omega_V(\tau^{\ast}{C} - E)) \\ %
 &\subseteq H^0(\Omega_V(\tau^{\ast}{C})) = H^0(Y, \Omega_{Y}(C)).
\end{split}
\end{equation*}
On the other hand it is not difficult to show that
\begin{equation*}
H^0(Y, \Omega_{Y}(C)) = H^0(\mathbb{CP}^1, \Omega_{\mathbb{CP}^1}(1)) = 0
\end{equation*}
by a similar method in Y. Lee and J. Park~\cite[Lemma~2]{Lee-Park-K^2=2}. Therefore the assertion follows.
\end{proof}

By the above lemma, we have the following commutative diagram of exact sequences:
%
\begin{equation}\label{equation:big-diagram}
{\tiny
\xymatrix{%
 &
 &
0 \ar[d]
 &
0 \ar[d]
& \\ %
 &
 &
H^0(\Omega_V \otimes \sheaf{O_E}(E)) \ar[d]^{\delta}
 &
H^0(\Omega_V(\log{F}) \otimes \sheaf{O_E}(E)) \ar[d]
 & \\ %
0 \ar[r]
 &
H^0(\sheaf{O_{F}}) \ar[r]  \ar[d]
 &
H^1(\Omega_V) \ar[r]^{\alpha_1} \ar[d]^{\beta_1}
 &
H^1(\Omega_V(\log{F})) \ar[r] \ar[d]^{\beta_2}
 &
0 \\ %
0 \ar[r]
 &
H^0(\sheaf{O_{F}}(E)) \ar[r]^{d_1'} \ar[dr]_{\gamma \circ d_1'}
 &
H^1(\Omega_V(E)) \ar[r]^{\alpha_2} \ar[d]^{\gamma}
 &
H^1(\Omega_V(\log{F})(E)) \ar[r] \ar[d]
 &
0 \\ %
 &
 &
H^1(\Omega_V \otimes \sheaf{O_E}(E)) \ar[d]
 &
H^1(\Omega_V(\log{F}) \otimes \sheaf{O_E}(E)) \ar[d]
 & \\ %
& & 0 & 0 &
}}
\end{equation}

\begin{lemma}\label{lemma:h^0=1}
$h^0(\Omega_{V}(\log{F}) \otimes \sheaf{O_E}(E)) = 1$.
\end{lemma}

\begin{proof}
Let
\begin{equation*}
\overline{r} : \Omega_V(\log{F}) \otimes \sheaf{O_E} \to \Omega_E \otimes \sheaf{O_E}(F)
\end{equation*}
be the restriction to the subsheaf $\Omega_V(\log{F})$ of the restriction map $r \otimes \id : \Omega_V|_E \otimes \sheaf{O_E}(F) \to \Omega_{E} \otimes \sheaf{O_E}(F)$.

\textit{Claim.} $\ker{\overline{r}} = \sheaf{O_E}(-E)$: Let $z_1$, $z_2$ be the local equations of $F$, $E$, respectively. Note that the restriction map $r : \Omega_V|_E \to \Omega_{E}$ is defined by
\begin{equation*}
r : \overline{a} \, dz_1 + \overline{b} \, dz_2 = \overline{a} \, d\overline{z_1}.
\end{equation*}
Hence the map $\overline{r}$ is defined by
\begin{equation*}
\overline{r} \left( \overline{a} \frac{d z_1}{z_1} + \overline{b} \, dz_2 \right) = \overline{a} \frac{d z_1}{z_1}.
\end{equation*}
It follows that $\ker{\overline{r}} = \ker{r} = \sheaf{O_E}(-E)$.

By the claim we have
\begin{equation*}
\ker[\overline{r} \otimes \id : \Omega_V(\log{F}) \otimes \sheaf{O_E}(E) \to \Omega_E \otimes \sheaf{O_E}(F + E)] = \sheaf{O_E}(-E) \otimes \sheaf{O_E}(E) = \sheaf{O_E}.
\end{equation*}
Hence we have the exact sequence
\begin{equation*}
0 \to H^0(\sheaf{O_E}) \to H^0(\Omega_V(\log{F}) \otimes \sheaf{O_E}(E)) \to H^0(\image(\overline{r} \otimes \id)).
\end{equation*}
Since $H^0(\image(\overline{r} \otimes \id)) \subset H^0(\Omega_E \otimes \sheaf{O_E}(F + E)) = 0$, the result follows.
\end{proof}

\begin{lemma}\label{lemma:surjective}
The composition map
\begin{equation*}
\gamma \circ d_1' : H^0(\sheaf{O_{F}}(E)) \to H^1(\Omega_{V} \otimes \sheaf{O_E}(E))
\end{equation*}
is surjective.
\end{lemma}

\begin{proof}
On the second column in \eqref{equation:big-diagram}, we have $h^0(\Omega_V \otimes \sheaf{O_E}(E)) = 1$, $h^1(\Omega_V) = 21$, $h^1(\Omega_V \otimes \sheaf{O_E}(E)) = 2$. Hence $h^1(\Omega_V(E)) = 22$. Since $h^0(\sheaf{O_{F}}) = 1$ and $h^0(\sheaf{O_{F}}(E)) = 3$, it follows from the first and the second columns that $h^1(\Omega_V(\log{F})) = 20$ and $h^1(\Omega_V(\log{F})(E)) = 19$. Therefore, on the third column, we have $h^1(\Omega_V(\log{F}) \otimes \sheaf{O_E}(E)) = 0$ because $h^0(\Omega_V(\log{F}) \otimes \sheaf{O_E}(E)) = 1$ by Lemma~\ref{lemma:h^0=1}. Hence $\beta_2$ is surjective.

Note that $H^0(\sheaf{O_{F}}) \subseteq \beta_1(H^1(\Omega_V)) \cap \ker{\alpha_2}$ in $H^1(\Omega_V(E))$. On the other hand, since $\alpha_1$ and $\beta_2$ are surjective, we have
\begin{equation*}
(\alpha_2 \circ \beta_1)(H^1(\Omega_V)) = (\beta_2 \circ \alpha_1)(H^1(\Omega_V)) = H^1(\Omega_V(\log{F})(E)).
\end{equation*}
Therefore $\dim[\beta_1(H^1(\Omega_V)) \cap \ker{\alpha_2}] = 1$; hence $H^0(\sheaf{O_{F}}) = \beta_1(H^1(\Omega_V)) \cap \ker{\alpha_2}$ in $H^1(\Omega_V(E))$. Thus
\begin{equation*}
\dim[\beta_1(H^1(\Omega_V)) \cap H^0(\sheaf{O_{F}}(E))] = 1.
\end{equation*}
Since $\ker{\gamma} = \image{\beta_1}$, we have $\dim[\ker{\gamma} \cap H^0(\sheaf{O_{F}}(E))] = 1$; hence the composition $\gamma \circ d_1'$ is surjective because $h^0(\sheaf{O_{F}}(E))=3$ but $h^1(\Omega_V \otimes \sheaf{O_E}(E)) = 2$.
\end{proof}

\begin{proof}[Proof of Proposition~\ref{proposition:H^2=0}]
Set $\mathbb{K} = H^0(\Omega_V(\log(D + S + F))(E))$. We want to show that $\mathbb{K} = 0$. Consider the commutative diagram
\begin{tiny}
\begin{equation*}
\xymatrix{%
& & & 0 \ar[d] \\ %
& & & H^0(\Omega_V \otimes \sheaf{O_E}(E)) \ar[d]^{\delta} \\ %
& 0 \ar[r] & \displaystyle  \bigoplus_{i=1}^{6} H^0(\sheaf{O_{D_i}}) \oplus \bigoplus_{i=1}^{3} H^0(\sheaf{O_{S_i}}) \oplus H^0(\sheaf{O_F}(E)) \ar[r]^(.7){c_1} \ar[d]& H^1(\Omega_V) \ar[d]^{\beta_1} \\ %
0 \ar[r] & \mathbb{K} \ar[r] & \displaystyle  \bigoplus_{i=1}^{6} H^0(\sheaf{O_{D_i}}) \oplus \bigoplus_{i=1}^{3} H^0(\sheaf{O_{S_i}}) \oplus H^0(\sheaf{O_F}(E)) \ar[r]^(.7){d_1} & H^1(\Omega_V(E)) \ar[d]^{\gamma} \\ %
& & & H^1(\Omega_V \otimes \sheaf{O_E}(E))
}
\end{equation*}
\end{tiny}

Let $\{1, v_1, v_2\}$ be a basis for $H^0(\sheaf{O_{F}}(E))$, where $1$ is a basis for $H^0(\sheaf{O_{F}}) \cong \mathbb{C}$. By Lemma~\ref{lemma:surjective} the images $\gamma \circ d_1(v_1)$ and $\gamma \circ d_1(v_2)$ span $H^1(\Omega_V \otimes \sheaf{O_E}(E))$. Therefore they are linearly independent in $H^1(\Omega_V(E))$ and they are not contained in the kernel of the map $\gamma$. On the other hand
\begin{equation*}
\bigoplus_{i=1}^{6} H^0(\sheaf{O_{D_i}}) \oplus \bigoplus_{i=1}^{3} H^0(\sheaf{O_{S_i}}) \oplus H^0(\sheaf{O_F}) \subset \ker{\gamma}.
\end{equation*}
Therefore the map $d_1$ is injective; hence $\mathbb{K} = \ker{d_1} = 0$.
\end{proof}

\section{Simply connected surfaces with $p_g=1$, $q=0$, and $1 \le K^2 \le 6$}
\label{section:Examples}

In this section we construct various examples of simply connected minimal surfaces of general type with $p_g=1$, $q=0$ and $1 \le K^2 \le 6$. Since all the proofs are basically the same as the case of the main construction, we describe only complex surfaces $Z$ which make it possible to get singular surfaces $X$ with permissible singularities.

\subsection{An example with $K^2=1$}

Let $Y$ be the K3 surface described in Section~\ref{section:construction}. We blow up the surface $Y$ totally three times at the three marked points $\bullet$; Figure~\ref{figure:K1}(A). We then get a surface $Z$; Figure~\ref{figure:K1}(B). There exist two disjoint linear chains of ${\mathbb{CP}}^1$'s in $Z$:
\begin{equation*}
\uc{-3}-\uc{-2}-\uc{-2}-\uc{-3}, \quad \uc{-2}-\uc{-5}-\uc{-3}.
\end{equation*}

\begin{figure}[hbtb]
\centering \subfloat[$Y$]{\includegraphics[scale=0.6]{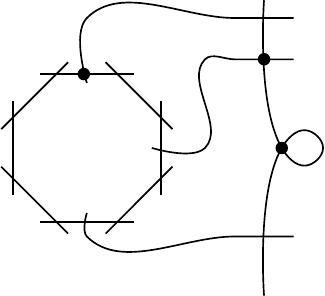}} \qquad
\subfloat[$Z = \blup{Y}{3}$]{\includegraphics[scale=0.67]{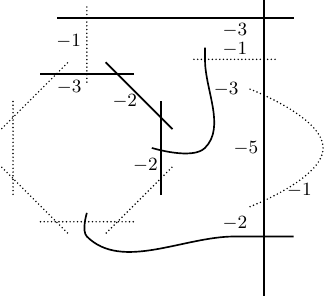}}
\caption{An example with $K^2=1$}
\label{figure:K1}
\end{figure}

\begin{remark}
According to Catanese~\cite{Catanese}, all minimal surfaces of general type with $p_g=1$ and $K^2=1$ are diffeomorphic and simply connected. Hence the example above is automatically simply connected. However we can prove the simply connectedness directly by using similar techniques -- rational blow-down surgery, Milnor fiber theory, and Van-Kampen Theorem -- in Y. Lee and J. Park~\cite{Lee-Park-K^2=2} and the authors~\cite{PPS-K3, PPS-K4}.
\end{remark}

\subsection{An example with $K^2=3$}

Let $A$, $L_i$ ($i=1,2,3$) be lines on the projective plane $\mathbb{CP}^2$ and $B$ a nonsingular conic on $\mathbb{CP}^2$ which intersect as in Figure~\ref{figure:K3-E}(A). Consider a pencil of cubics generated by the two cubics $A+B$ and $L_1+L_2+L_3$. Blow up the five base points $\bullet$ of the pencil of cubics including infinitely near base-points at each point. Then we obtain a rational elliptic surface $E(1)$ with an $I_6$-singular fiber, $I_3$-singular fiber, three nodal singular fibers, and five sections which intersect as in Figure~\ref{figure:K3-E}(B), where we omit one nodal singular fiber and one section which are not used in the following construction.

\begin{figure}[hbtb]
\centering \subfloat[A pencil]{\includegraphics[scale=1]{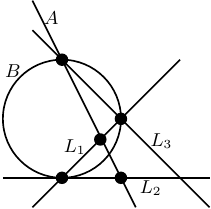}} \qquad
\subfloat[$E(1)$]{\includegraphics[scale=0.67]{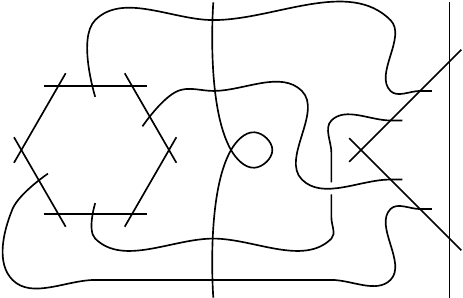}}
\caption{A rational elliptic surface $E(1)$ for $K^2=3$}
\label{figure:K3-E}
\end{figure}

Let $Y$ be a double cover of the rational elliptic surface $E(1)$ branched along two general fibers. Then $Y$ is an elliptic K3 surface with two $I_6$-singular fibers, two $I_3$-singular fibers, six nodal singular fibers, and five sections. We use only two $I_6$-singular fibers, one $I_3$-singular fibers, one nodal singular fibers, and four sections; Figure~\ref{figure:K3-sc}(A).

We blow up the surface $Y$ six times at the six marked points $\bullet$ and twice at the marked point $\bigodot$. We then get a surface $Z$; Figure~\ref{figure:K3-sc}(B). There exist four disjoint linear chains of ${\mathbb{CP}}^1$'s in $Z$:
\begin{equation*}
\uc{-5}-\uc{-2}-\uc{-6}-\uc{-2}-\uc{-2}-\uc{-2}, \quad \uc{-2}-\uc{-3}-\uc{-4}, \quad \uc{-2}-\uc{-3}-\uc{-4}, \quad \uc{-3}-\uc{-3}.
\end{equation*}

\begin{figure}[hbtb]
\centering \subfloat[$Y$]{\includegraphics[scale=0.6]{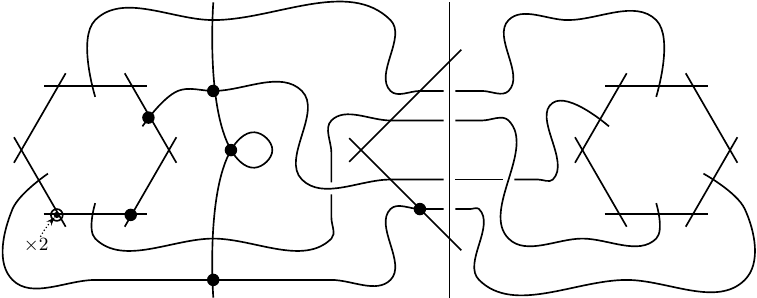}} \\
\subfloat[$Z = \blup{Y}{8}$]{\includegraphics[scale=0.67]{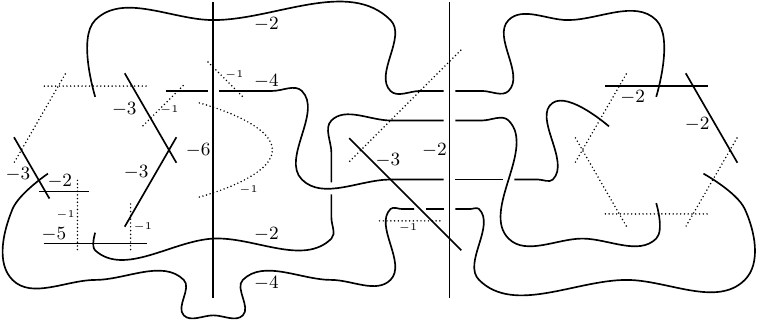}}
\caption{An example with $K^2=3$}
\label{figure:K3-sc}
\end{figure}

\subsection{An example with $K^2=4$}

Let $A$ and $L$ be lines on the projective plane $\mathbb{CP}^2$ and $B$ a nonsingular conic on $\mathbb{CP}^2$ which intersect as in Figure~\ref{figure:K4-sc-E}(A). Consider a pencil of cubics generated by the two cubics $A+B$ and $3L$. Blow up the three base points of the pencil of cubics including infinitely near base-points at each point. Then we obtain a rational elliptic surface $E(1)$ with an $\tilde{E}_6$-singular fiber, an $I_2$-singular fiber, two nodal singular fibers, and three sections; Figure~\ref{figure:K4-sc-E}(B).

\begin{figure}[hbtb]
\centering
\subfloat[A pencil]{\includegraphics[scale=1]{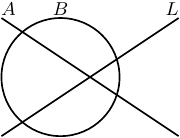}} \qquad
\subfloat[$E(1)$]{\includegraphics[scale=0.67]{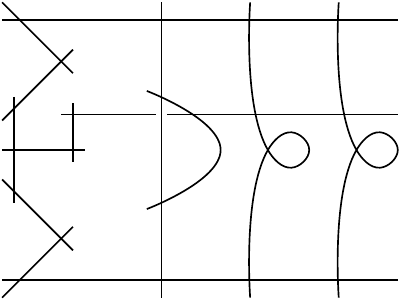}}
\caption{A rational elliptic surface $E(1)$ for $K^2=4$}
\label{figure:K4-sc-E}
\end{figure}

Let $Y$ be a double cover of the rational elliptic surface $E(1)$ branched along two general fibers. Then $Y$ is an elliptic K3 surface with two $\tilde{E}_6$-singular fibers, two $I_2$-singular fibers, four nodal singular fibers, and three sections. We use only two $\tilde{E}_6$-singular fibers, two $I_2$-singular fibers, one nodal singular fibers, and three sections; Figure~\ref{figure:K4-sc}(A).

We blow up the surface $Y$ totally $16$ times at the marked points. We then get a surface $Z$; Figure~\ref{figure:K4-sc}(B). There exist six disjoint linear chains of ${\mathbb{CP}}^1$'s in $Z$:
\begin{gather*}
\uc{-2}-\uc{-10}-\uc{-2}-\uc{-2}-\uc{-2}-\uc{-2}-\uc{-2}-\uc{-3}, \quad \uc{-2}-\uc{-8}-\uc{-2}-\uc{-2}-\uc{-2}-\uc{-3} \\
\uc{-3}-\uc{-2}-\uc{-2}-\uc{-3}, \quad \uc{-2}-\uc{-5}-\uc{-3}, \quad \uc{-4}, \quad \uc{-4}.
\end{gather*}

\begin{figure}[hbtb]
\centering \subfloat[$Y$]{\includegraphics[scale=0.6]{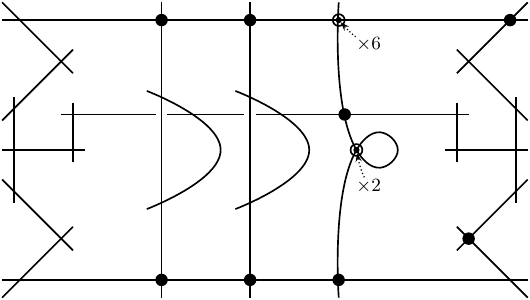}} \\
\subfloat[$Z = \blup{Y}{16}$]{\includegraphics[scale=0.67]{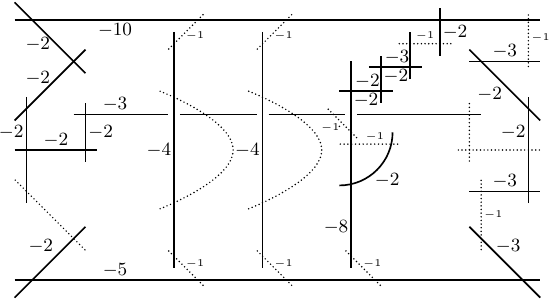}}
\caption{An example with $K^2=4$}
\label{figure:K4-sc}
\end{figure}

\subsection{An example with $K^2=5$}

Let $Y$ be the K3 surface used in Section~\ref{sec:h_1=z_2}. We use only two $I_7$-singular fibers, two $I_2$-singular fibers, one nodal singular fibers, and three sections; Figure~\ref{figure:K5-sc}(A).

\begin{figure}[hbtb]
\centering \subfloat[$Y$]{\includegraphics[scale=0.6]{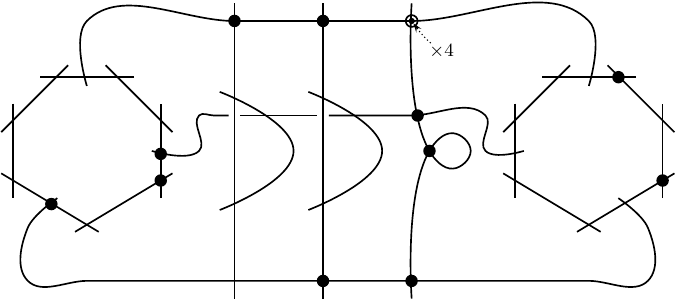}} \\
\subfloat[$Z = \blup{Y}{15}$]{\includegraphics[scale=0.66]{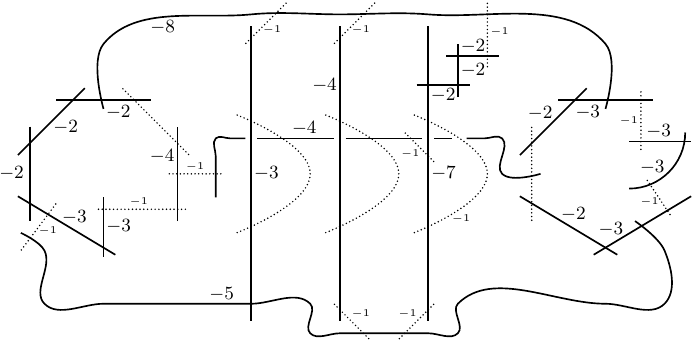}}
\caption{An example with $K^2=5$}
\label{figure:K5-sc}
\end{figure}

We blow up the surface $Y$ totally $15$ times at the marked points. We then get a surface $Z$; Figure~\ref{figure:K5-sc}(B). There exist seven disjoint linear chains of ${\mathbb{CP}}^1$'s in $Z$:
\begin{gather*}
\uc{-2}-\uc{-3}-\uc{-8}-\uc{-2}-\uc{-2}-\uc{-2}-\uc{-3}-\uc{-3}, \quad \uc{-7}-\uc{-2}-\uc{-2}-\uc{-2} \\
\uc{-3}-\uc{-5}-\uc{-3}-\uc{-2}, \quad \uc{-3}-\uc{-3}, \quad \uc{-4}, \quad \uc{-4}, \quad \uc{-4}.
\end{gather*}

\subsection{An example with $K^2=6$}

Let $Y$ be the elliptic K3 surface described in Section~\ref{section:construction}. We use only two $I_8$-singular fibers, two $I_2$-singular fibers, one nodal singular fibers, and three sections; Figure~\ref{figure:K6-sc}(A).

We blow up the surface $Y$ totally $18$ times at the marked points. We then get a surface $Z$; Figure~\ref{figure:K6-sc}(B). There exist five disjoint linear chains of ${\mathbb{CP}}^1$'s in $Z$:
\begin{gather*}
\uc{-2}-\uc{-2}-\uc{-3}-\uc{-9}-\uc{-2}-\uc{-2}-\uc{-2}-\uc{-2}-\uc{-3}-\uc{-4}, \\
\uc{-2}-\uc{-3}-\uc{-7}-\uc{-2}-\uc{-2}-\uc{-3}-\uc{-3}, \quad \uc{-7}-\uc{-2}-\uc{-2}-\uc{-2}, \\
\uc{-4}-\uc{-3}-\uc{-2}, \quad \uc{-4}.
\end{gather*}

\begin{figure}[hbtb]
\centering \subfloat[$Y$]{\includegraphics[scale=0.6]{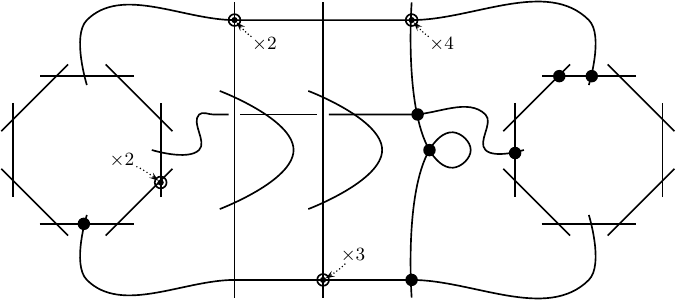}} \\
\subfloat[$Z = \blup{Y}{18}$]{\includegraphics[scale=0.663]{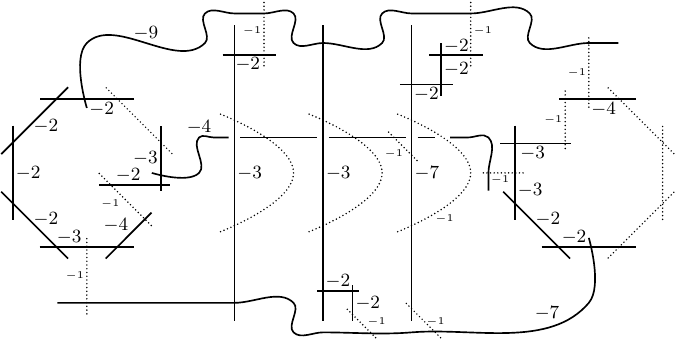}}
\caption{An example with $K^2=6$}
\label{figure:K6-sc}
\end{figure}

\section{A simply connected surface with $p_g=1$, $q=0$, and $K^2=8$}
\label{section:K^2=8}

In this section we construct a simply connected minimal complex surface of general type with $p_g=1$, $q=0$, and $K^2=8$.

According to Kondo~\cite{Kondo}, there is an Enriques surface $Y$ with an elliptic fibration over $\mathbb{CP}^1$ which has a $I_9$-singular fiber, a nodal singular fiber $F$, and two bisections $S_1$ and $S_2$; Figure~\ref{figure:Y}. Again by Kondo~\cite{Kondo}, there is an unbranched double covering $\pi: \overline{Y} \to Y$ of $Y$ where $\overline{Y}$ is an elliptic K3 surface  which has two $I_9$-singular fiber, two nodal singular fiber $\overline{F}_1$ and $\overline{F}_2$, and four sections $\overline{S}_1, \dotsc, \overline{S}_4$ such that $\pi(\overline{F}_1) = \pi(\overline{F}_2) = F$, $\pi(\overline{S}_1) = \pi(\overline{S}_3) = S_1$, and $\pi(\overline{S}_2) = \pi(\overline{S}_4) = S_2$; Figure~\ref{figure:Ybar}.

\begin{figure}[hbtb]
\centering
\includegraphics{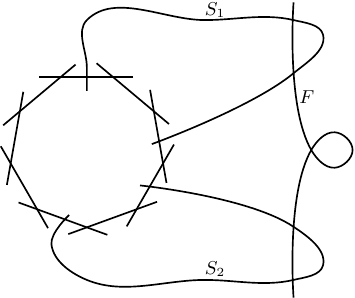}
\caption{An Enriques surface $Y$}
\label{figure:Y}
\end{figure}

\begin{figure}[hbtb]
\centering
\includegraphics{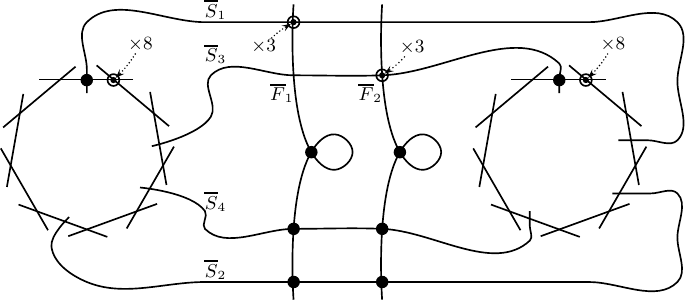}
\caption{A K3 surface $\overline{Y}$}
\label{figure:Ybar}
\end{figure}

We blow up the K3 surface $\overline{Y}$ totally 30 times at the marked points $\bullet$ and $\bigodot$. We then get a surface $\overline{Z} = \blup{\overline{Y}}{30}$; Figure~\ref{figure:Zbar}. There exist four disjoint linear chains of $\mathbb{CP}^1$'s in $\overline{Z}$:
\begin{align*}
&C_{19,6}: \uc{-2}-\uc{-2}-\uc{-9}-\uc{-2}-\uc{-2}-\uc{-2}-\uc{-2}-\uc{-4} \\
&C_{19,6}: \uc{-2}-\uc{-2}-\uc{-9}-\uc{-2}-\uc{-2}-\uc{-2}-\uc{-2}-\uc{-4} \\
&C_{73,50}: \uc{-2}-\uc{-2}-\uc{-7}-\uc{-6}-\uc{-2}-\uc{-3}-\uc{-2}-\uc{-2}-\uc{-2}-\uc{-2}-\uc{-4}\\
&C_{73,50}: \uc{-2}-\uc{-2}-\uc{-7}-\uc{-6}-\uc{-2}-\uc{-3}-\uc{-2}-\uc{-2}-\uc{-2}-\uc{-2}-\uc{-4}
\end{align*}

\begin{figure}[hbtb]
\centering
\includegraphics{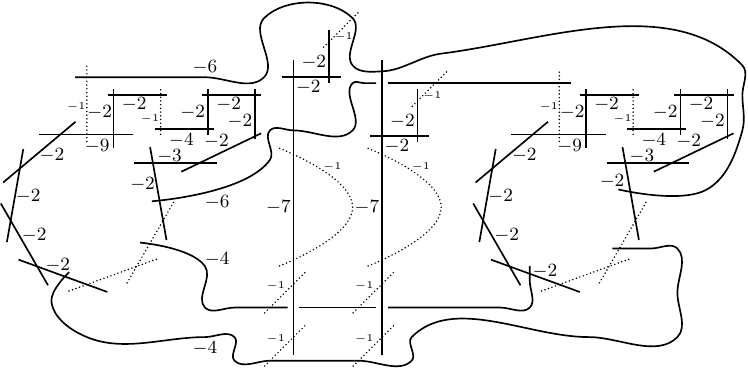}
\caption{A surface $\overline{Z} = \blup{\overline{Y}}{30}$}
\label{figure:Zbar}
\end{figure}

We contract these four chains of $\mathbb{CP}^1$'s from the surface $\overline{Z}$ so that it produces a normal projective surface $\overline{X}$ with four singular points of class $T$. It is not difficult to show that $H^2(X, \sheaf{T_X}) \neq 0$.

\begin{theorem}\label{theorem:Q-Gorenstein}
The singular surface $\overline{X}$ has a global $\mathbb{Q}$-Gorenstein smoothing. A general fiber $\overline{X}_t$ of the smoothing of $\overline{X}$ is a simply connected minimal complex surface of general type with $p_g=1$, $q=0$, and $K^2=8$.
\end{theorem}

In order to prove Theorem~\ref{theorem:Q-Gorenstein}, we apply the following proposition.

\begin{proposition}[{Y. Lee and J. Park~\cite{Lee-Park-Horikawa}}]\label{proposition:Lee-Park}
Let $V$ be a normal projective surface with singularities of class $T$. Assume that a cyclic group $G$ acts on $X$ such that
\begin{enumerate}[1.]
\item $W = V/G$ is a normal projective surface with singularities of class $T$,

\item $p_g(W)=q(W)=0$,

\item $W$ has a $\mathbb{Q}$-Gorenstein smoothing,

\item the map $\sigma : V \to W$ induced by a cyclic covering is flat, and the branch locus $D$ (resp. the ramification locus) of the map $\sigma : V \to W$ is a nonsingular curve lying outside the singular locus of $W$ (resp. of $V$), and

\item $H^1(W, \sheaf{O_W}(D))=0$.
\end{enumerate}
Then there exists a $\mathbb{Q}$-Gorenstein smoothing of $V$ that is compatible with a $\mathbb{Q}$-Gorenstein smoothing of $W$. Furthermore the cyclic covering extends to the $\mathbb{Q}$-Gorenstein smoothing.
\end{proposition}

We now construct an unramified double covering from the singular surface $\overline{X}$ to another singular surface. We begin with the Enriques surface $Y$ in Figure~\ref{figure:Y}. We blow up totally 15 times at the marked points $\bullet$ and $\bigodot$; Figure~\ref{figure:Y-dots}. We then get a surface $Z = \blup{Y}{15}$; Figure~\ref{figure:Z}. There exist two disjoint linear chains of $\mathbb{CP}^1$'s in $Z$:
\begin{align*}
&C_{19,6}: \uc{-2}-\uc{-2}-\uc{-9}-\uc{-2}-\uc{-2}-\uc{-2}-\uc{-2}-\uc{-4} \\
&C_{73,50}: \uc{-2}-\uc{-2}-\uc{-7}-\uc{-6}-\uc{-2}-\uc{-3}-\uc{-2}-\uc{-2}-\uc{-2}-\uc{-2}-\uc{-4}
\end{align*}

\begin{figure}[hbtb]
\centering
\includegraphics{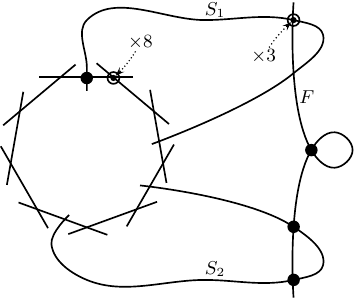}
\caption{An Enriques surface $Y$ with marked points}
\label{figure:Y-dots}
\end{figure}

\begin{figure}[hbtb]
\centering
\includegraphics{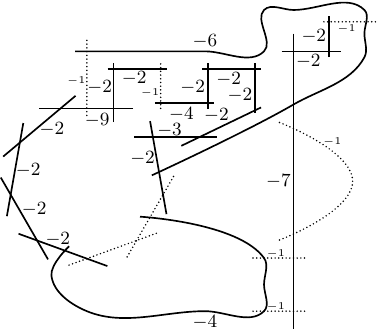}
\caption{A surface $Z = \blup{Y}{15}$}
\label{figure:Z}
\end{figure}

We contract the two chains of $\mathbb{CP}^1$'s from the surface $Z$ so that it produces a normal projective surface $X$ with two singular points class $T$. It is clear that there is an unbranched double covering $\overline{\pi} : \overline{X} \to X$. The singular surface $X$ satisfies the third condition of Proposition~\ref{proposition:Lee-Park}.

\begin{proposition}[{H. Park~\cite{Park}}]\label{proposition:Park}
The singular surface $X$ has a global $\mathbb{Q}$-Gorenstein smoothing. A general fiber $X_t$ of the smoothing of $X$ is a minimal complex surface of general type with $p_g=0$, $K^2=4$, and $\pi_1(X_t) = \mathbb{Z}/2\mathbb{Z}$.
\end{proposition}

\begin{proof}[Proof of Theorem~\ref{theorem:Q-Gorenstein}]
It is easy to show that the covering $\overline{\pi} : \overline{X} \to X$ satisfies all conditions of Proposition~\ref{proposition:Lee-Park}. Therefore the singular surface $\overline{X}$ has a global $\mathbb{Q}$-Gorenstein smoothing. Let $\overline{X}_t$ be a general fiber of the smoothing of $\overline{X}$. Since $p_g(\overline{X})=1$, $q(\overline{X})=0$, and $K_{\overline{X}}^2=8$, by applying general results of complex surface theory and $\mathbb{Q}$-Gorenstein smoothing theory, one may conclude that a general fiber $\overline{X}_t$ is a complex surface of general type with $p_g=1$, $q=0$, and $K^2=8$. Furthermore, it is not difficult to show that a general fiber $X_t$ is minimal by using a similar technique in Y. Lee and J. Park~\cite{Lee-Park-K^2=2} and the authors~\cite{PPS-K3, PPS-K4}.

\textit{Claim.} A general fiber $\overline{X}_t$ is simply connected: By Proposition~\ref{proposition:Lee-Park}, there is an induced unbranched double covering $\overline{X}_t \to X_t$; hence, a general fiber $\overline{X}_t$ is simply connected because $\pi(X_t) = \mathbb{Z}/2\mathbb{Z}$ by Proposition~\ref{proposition:Park}.
\end{proof}

\providecommand{\bysame}{\leavevmode\hbox to3em{\hrulefill}\thinspace}

\end{document}